\documentclass[11pt]{amsart}
\hoffset= - 0.75 in

\usepackage{amsfonts, amsmath, amssymb, amsthm}
\usepackage{latexsym, graphicx, color}
\newtheorem{thm}{Theorem}[section]

\newtheorem{conj}[thm]{Conjecture}
\newtheorem{defn}[thm]{Definition}

\newtheorem{lem}[thm]{Lemma}

\newtheorem{prop}[thm]{Proposition}

\newcommand{\RR}{\mathbb{R}}

\newcommand{\TP}{\mathbb{TP}}

\begin{document}

\title{Tropical hyperplane arrangements and oriented matroids}
\author{Federico Ardila and Mike Develin}
\address{Federico Ardila, San Francisco State University 1600 Holloway Ave., San Francisco, CA, USA\\ Mike Develin, American Institute of Mathematics, 360 Portage 
Ave., Palo Alto, CA, USA}
\date{\today}
\email{federico@math.sfsu.edu, develin@post.harvard.edu}

\begin{abstract}
We study the combinatorial properties of a tropical hyperplane arrangement. We define tropical oriented matroids, and prove that they share many of the properties of ordinary oriented matroids.
We show that a tropical oriented matroid determines a subdivision of a product of two simplices, and conjecture that this correspondence is a bijection. 
\end{abstract}

\maketitle

\section{Introduction}

Tropical mathematics is the study of the tropical semiring consisting of the real numbers with the operations of $+$ and $max$. This semiring can be thought of as the image of a power series ring under the 
degree map which sends a power series to its leading exponent.
This semiring has received great attention recently in several areas of mathematics, due to the discovery that there are often strong relationships between a classical question and its tropical counterpart. 
One can then translate geometric questions about algebraic varieties into combinatorial questions about polyhedral fans. 
This point of view has been fruitful in algebraic geometry, combinatorics, and phylogenetics, among others
 \cite{AK, ARW, D, DS, M, RST, SS}.

The triangulations of a
product of two simplices are ubiquitous and useful objects.
They are of
independent interest \cite{Babson, Bayer, Gelfand, Postnikov}, and have been
used as a building block for finding efficient triangulations of
high dimensional cubes \cite{Haiman, Orden} and disconnected
flip-graphs \cite{Santosflips, Santostoric}. They also arise very
naturally in connection with the Schubert calculus \cite{AB},
Hom-complexes \cite{Pfeifle}, growth series of root lattices \cite{Seashore}, transportation problems and Segre embeddings \cite{Sturmfels}, among others.

The goal of this paper is to start laying down the foundations of a theory of tropical oriented matroids. In the same way that oriented matroids capture the combinatorial properties of real hyperplane arrangements and ordinary polytopes, these objects are modeled after tropical hyperplane arrangements and tropical polytopes. 
We present strong evidence of the intimate connection between them and the subdivisions of a product of two simplices: a tropical oriented matroid determines a subdivision, and we conjecture that this is a bijection. 
We expect that further development of the theory will lead to elegant structural results and 
applications in the numerous areas where these objects appear; we present several results and conjectures to that effect.

\medskip

The paper is organized as follows. In Section \ref{sec:defs} we recall some background information on tropical geometry. Section \ref{sec:tropoms} defines tropical oriented matroids, and proves that every tropical hyperplane arrangement gives rise to one. In Section \ref{sec:properties} we prove that a tropical oriented matroid is completely determined by its topes (maximal elements), and is also determined by its vertices (minimal elements). We define the notions of deletion and contraction. Section \ref{sec:conjectures} presents three key conjectures for tropical oriented matroids: a bijection with subdivisions of a product of two simplices, a notion of duality,
and a topological representation theorem. We then show the potential applications to two open problems in the literature. Finally, Section \ref{sec:2d} shows that a tropical oriented matroid does determine such a subdivision, and prove the reverse direction for triangulations in the two-dimensional case.



\section{Basic definitions}\label{sec:defs}
In this section we recall some basic definitions from tropical geometry. For more
information, see \cite{DS, SS}.

\begin{defn}
The \textbf{tropical semiring} is given by the real numbers $\RR$ together with the 
operations of tropical addition $\oplus$ and tropical multiplication $\odot$ 
defined by $a\oplus b = \text{max}(a, b)$ and $a\odot b = a + b$.
\end{defn}

This tropical semiring can be 
thought of as the image of ordinary arithmetic in a power series ring under the 
degree map, which sends a power series in $t^{-1}$ to its leading exponent. 
As in ordinary geometry, 
we can form tropical $d$-space, $\RR^d$ with the operations of vector addition 
(coordinatewise maximum) and scalar multiplication (adding a constant to each 
vector.) For many purposes, it proves convenient to work in tropical projective 
$(d-1)$-space $\TP^{d-1}$, given by modding out by tropical scalar multiplication; 
this produces the ordinary vector space quotient $\RR^d / (1, \ldots, 1)\RR$, which 
can be depicted as real $(d-1)$-space.

In this space, one important class of objects is tropical hyperplanes. These are
given by the vanishing locus of a single linear functional $\bigoplus c_i\odot 
x_i$; in tropical mathematics, this vanishing locus is defined to be the set of 
points where the encoded maximum $\text{max}(c_1+x_1,\ldots, c_d+x_d)$ is 
achieved at least twice. (Reflection on the power series etymology of tropical 
mathematics will yield the motivation for this definition.)

These tropical hyperplanes are given by fans polar to the simplex formed by the 
standard basis vectors $\{e_1, \ldots, e_d\}$; the apex of $\bigoplus c_i \odot 
x_i$ is $(-c_1, \ldots, -c_d)$. Each of these fans has a natural index on 
each of its cones: the subset of $[d]$ for which $c_i + x_i$ is maximized. On the 
$d$ full-dimensional sectors, this is a singleton; on the cone polar to a subset of 
basis vectors forming a face of the simplex, it is given by that subset of 
coordinates. Figure \ref{types} shows what tropical hyperplanes look
like in $\TP^2$, where the point $(a,b,c)$ in $\TP^2$ is represented by the 
point $(0,b-a,c-a)$ in $\RR^2$.

Another natural class of geometric objects in tropical mathematics is that of 
tropical polytopes.

\begin{defn}
Given a set of points $V=\{v_1,\ldots,v_n\} \subset \TP^{d-1}$, their \textbf{tropical convex 
hull} is 
the set of all (tropical) linear combinations\footnote{Note that we take 
{\bf all} linear combinations, not just the nonnegative ones adding up to $1$ as in regular convexity.} $\bigoplus c_i \odot v_i$ with $c_i \in \RR$, where the scalar multiplication $c_i \odot v_i$ is defined componentwise.
A \textbf{tropical polytope} is the 
tropical convex hull of a finite set of points. 
\end{defn}

Tropical polytopes are bounded polyhedral complexes~\cite{DS}. An important theorem 
connects tropical polytopes with tropical hyperplane arrangements:

\begin{thm}\cite{DS}\label{TP=THA}
Let $P$ be the tropical convex hull of a finite point set $V = \{v_1, \ldots, 
v_n\}$. Then $P$ is the union of the bounded regions of the polyhedral 
decomposition of $\TP^{d-1}$ given by putting an inverted hyperplane at each point 
$v_1, \ldots, v_n$.
\end{thm}

This inverted hyperplane arrangement is of course combinatorially equivalent to a 
hyperplane arrangement given by hyperplanes with apexes $\{-v_1, \ldots,-v_n\}$. 
Indeed, both of these are given by the regular subdivisions of a product of simplices; see \cite{Sturmfels, Ziegler} for more information on
this topic.

\begin{thm}~\cite{DS}
The convex hull of a finite point set $V = \{v_1, \ldots, v_n\}$, with the polyhedral subdivision given by 
Theorem~\ref{TP=THA} is combinatorially isomorphic to the complex of interior faces of the regular subdivision of 
$\Delta^{n-1}\times \Delta^{d-1}$, where the height of vertex $(i, j)$ is given by the $j$-th coordinate of $v_i$. The 
corresponding tropical hyperplane arrangement is isomorphic to the complex of 
faces which contain at least 
one vertex from each of the $n$ copies of $\Delta^{d-1}$.
\end{thm}

Indeed, we can say precisely 
which face each region of the hyperplane arrangement corresponds to; see the 
discussion of Section \ref{sec:conjectures}.

The combinatorial structure of a tropical hyperplane arrangement is captured by its collection of \textbf{types}. Given an arrangement $H_1,\ldots,H_n$ in $\TP^{d-1}$, the \textbf{type} of a point $x\in \TP^{d-1}$ is the $n$-tuple $(A_1, \ldots, A_n)$, where $A_i \subseteq [d]$ is the set of closed sectors of the hyperplane $H_i$ which $x$ is contained in.\footnote{Notice that this definition of type is the transpose of the definition in \cite{DS}.} Algebraically, if hyperplane $H_i$ has vertex $v_i=(v_{i1},\ldots,v_{id})$, the set $A_i$ consists of the indices $j$ for which $x_j-v_{ij}$ is maximal.

Since all the points in a face of the arrangement have the same type, we call this the type of the face. Figure \ref{types} shows an arrangement of three tropical hyperplane arrangements in $\TP^2$, and indicates the types of some of the faces. The labelling of the sectors of the hyperplanes is indicated on the right.

\begin{figure}[h]
\begin{center}\includegraphics[height=8cm]{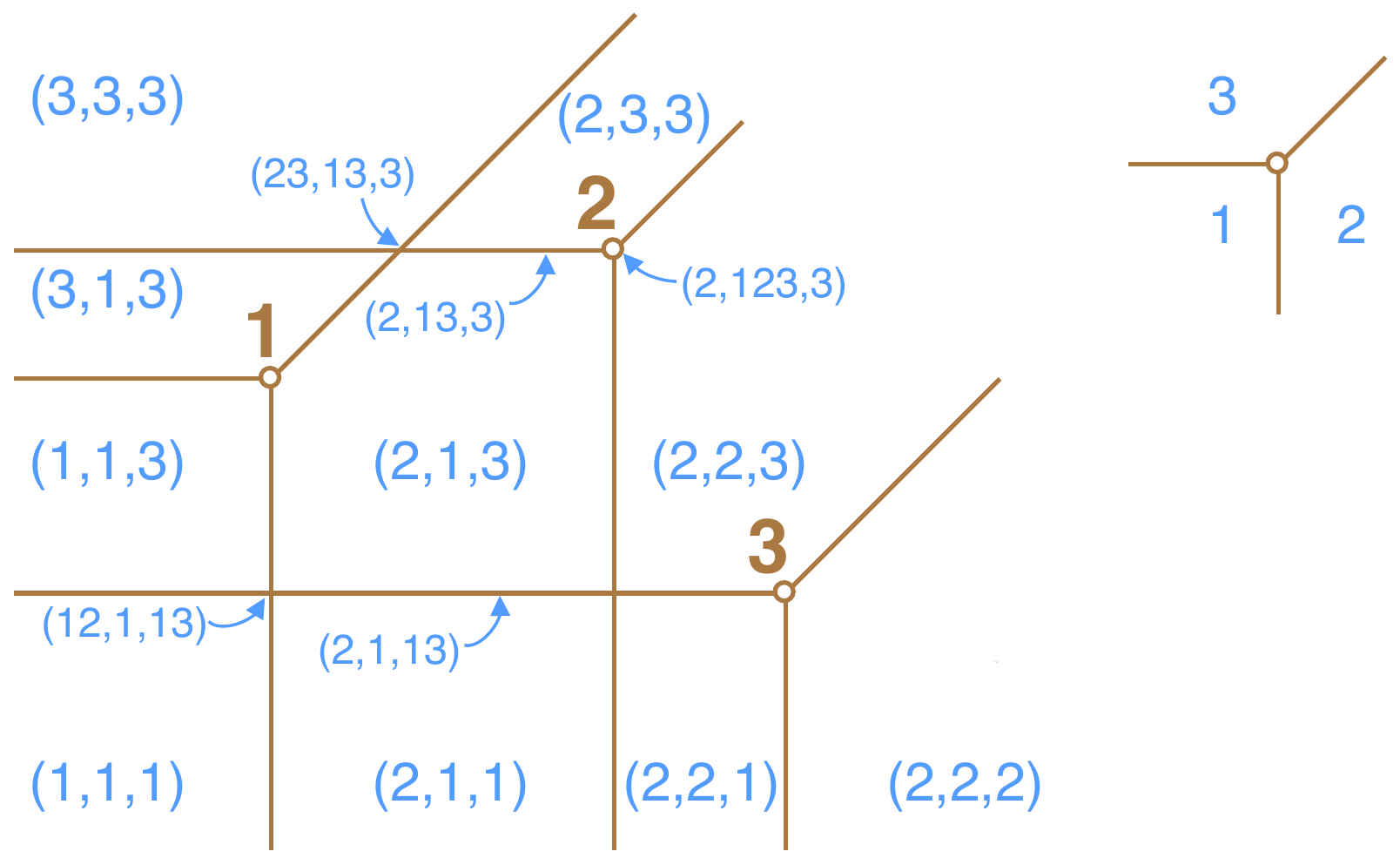}\end{center}
\caption{\label{types} An arrangement in $\TP^2$ and the types of some of its faces.}
\end{figure}

\section{Tropical oriented matroids.}\label{sec:tropoms}





Our goal will be to define tropical oriented matroids as a collection of types
(analogous to the covectors of an ordinary oriented matroid) which satisfy  
certain axioms inspired by tropical hyperplane arrangements as well as by ordinary 
oriented matroids. We proceed with some preliminary definitions.

\begin{defn}
An \textbf{$(n, d)$-type} is an $n$-tuple $(A_1,\ldots, A_n)$ of nonempty subsets of $[d]:=\{1,\ldots,d\}$.
\end{defn}

\begin{defn}
Given two $(n,d)$-types $A$ and $B$, the \textbf{comparability graph} $CG_{A, 
B}$ has vertex set $[d]$. For $1\leq i \leq n$, we draw an edge between $j$ and $k$ for each $j \in A_i$ and $k \in B_i$. That edge is undirected if $j,k \in A_i \cap B_i$, and it is directed $j \rightarrow k$ otherwise.
\end{defn}

This object is a \textbf{semidigraph}:

\begin{defn}
A \textbf{semidigraph} is a graph with some undirected edges and some 
directed edges. A \textbf{directed path} from $a$ to $b$ in a semidigraph 
is a collection of vertices $v_0 = a, v_1, \ldots, v_k = b$ and a 
collection of edges $e_1, \ldots, e_k$, at least one of which is directed, such that $e_i$ is either a 
directed edge from $v_{i-1}$ to $v_i$ or an undirected edge connecting the 
two. A \textbf{directed cycle} is 
a directed path with identical endpoints. A semidigraph is 
\textbf{acyclic} if it has no directed cycles.
\end{defn}

A property that a point lying in a tropical hyperplane arrangement should have is that its type should locally change in a predictable way. The 
next definition will help us rigorize this.

\begin{defn}

The \textbf{refinement} of a type $A = (A_1, \ldots, A_n)$ with respect to an ordered partition $(P_1, \ldots, P_r)$ of $[d]$ is $A = (A_1 \cap P_{m(1)}, \ldots, A_n \cap P_{m(n)})$
where $m(i)$ is the largest index for which $A_i \cap P_{m(i)}$ is non-empty.
A refinement is \textbf{total} if 
each $B_i$ is a singleton.
\end{defn}


With these definitions, we are ready to give a natural axiomatic definition of a tropical oriented matroid.

\begin{defn}
A \textbf{tropical oriented matroid} $M$ (with parameters $(n, d)$) is a collection of 
$(n, d)$-types which satisfy the following four axioms:
\begin{itemize}
\item Boundary: For each $j\in [d]$, the type $\textbf{j} := (j,j, \ldots,j)$ is in $M$.

\item Elimination: If we have two types $A$ and $B$ in $M$ and a position $j\in 
[n]$, then there exists a type $C$ in $M$ with $C_j = A_j\cup B_j$, and $C_k \in \{A_k, B_k, A_k\cup B_k\}$ for all $k\in [n]$.

\item Comparability: The comparability graph $CG_{A,B}$ of any two types $A$ and $B$ in $M$ is acyclic.

\item Surrounding: If $A$ is a type in $M$, then any refinement of $A$ is also in $M$.

\end{itemize}
\end{defn}

The boundary axiom comes from the fact that tropical hyperplanes are all translates of each other; if we go off to infinity in a given basis direction, then we end up in that sector of each hyperplane. 

The elimination axiom roughly tells us how to go from a more general to a more special position with respect to a hyperplane. It predicts what happens when we intersect the $j$th hyperplane as we walk from $A$ to $B$ along a tropical line segment.
For ordinary oriented matroids, if we have two covectors which have opposite signs in some coordinate, then we can produce a covector which has a 0 in that coordinate and agrees with $A$ and $B$ whenever these two agree. Our elimination axiom is just the $d$-sign 
version of this; if we replace the set $[d]$ with the signs $\{+. -\}$, using the 2-element set $\{+, -\}$ for 0, then the elimination axiom yields 
the normal oriented matroid elimination axiom. Note that a type element $A_i$ being larger corresponds to it being more like 0; the ultimate vanishing point 
in the hyperplane is its apex $v_i$, where $A_i = [d]$.

An edge from $j$ to $k$ in the comparability graph $CG_{A,B}$ roughly indicates that, as you walk from region $A$ to region $B$, you are moving more in the $k$ direction than in the $j$ direction. Therefore this graph cannot have cycles.

The composition axiom tells us how to go from a more special to a more general position. In ordinary oriented matroids, for any two covectors $A$ and $B$, we can find a covector which agrees with $A$ everywhere except where the sign of $A$ is 0, when it agrees with the sign of $B$. This corresponds to moving infinitesimally from $A$ 
towards $B$; essentially, doing this process for all $B$ yields local information about signs around $A$, where $B$-covectors are proxies for 
directions. Because tropical hyperplanes are all translates of each other, we don't need any directional information from $B$; refining a type 
corresponds to moving infinitesimally away from the point in a direction indicated by an ordered partition. In other words, we can write down the types 
of points in a local neighborhood of $A$ simply by looking at $A$ itself, which is what the surrounding axiom does.

Obviously, if this definition is to have any merit, the following must be true.

\begin{thm}\label{thm:arr.is.om}
The collection of types in a tropical hyperplane arrangement forms a tropical oriented matroid.
\end{thm}

\begin{proof}
Let the apexes of the arrangement be given by $\{v_1, \ldots, v_n\}$. We verify the axioms in the order given above.

\textit{Boundary}: Taking a point with $x_j$ large enough (specifically, such that $x_j-x_i > v_{kj} - v_{ki}$ for 
all $k\in [n]$ and $i\neq j\in [d]$) works.

\textit{Elimination}: Suppose we have points $x$ and $y$ of types $A$ and $B$, and a position $j\in [n]$. 
Let $a \in A_j$ and $b \in B_j$, and pick coordinates for $x$ and $y$ (by adding a multiple of $(1, 
\ldots, 1)$) such that $x_a - v_{ja} = y_b - v_{jb} = 0$. Consider the point $z = x\oplus y$, the coordinatewise maximum of $x$ and $y$. 
Then $x_i - v_{ji} = 0$ for $i\in A_j$, and this maximizes this difference over all $i$; similarly, $y_i - v_{ji} = 0$ for $i\in B_j$, and 
this maximizes this difference over all $i$. Therefore $z_i - v_{ji}$ is $0$ for $i\in A_j \cup B_j$, and it is negative for other $i$. Let $C$ be the type of $z$; then $C_j = A_j\cup B_j$. 

Now, consider any $k\neq j$. We have $z_i - v_{ki} = \text{max} (x_i - v_{ki}, y_i - v_{ki})$; we need to find the values of $i$ for which this is maximized. 
The maximum value of this is equal to $\text{max} (\text{max}_i (x_i - v_{ki}), \text{max}_i (y_i - v_{ki}))$. If the first maximum is bigger, then 
the values of $i$ which maximize $z_i - v_{ki}$ are precisely those which maximize $x_i - v_{ki}$, \emph{i.e.} $C_k = A_k$. Similarly, if the second maximum is 
bigger, then we have $C_k = B_k$. Finally, if the maxima are the same, then we have $C_k = A_k\cup B_k$. 

Essentially, we have formed the tropical 
line segment between $x$ and $y$, and taken the point on it which is on the maximal cone of $H_j$. Every 
point on this line segment has $C_k\in \{A_k, B_k, A_k\cup B_k\}$ for all $k$.

\textit{Comparability}:
Let $A$ and $B$ be types realized by points $x$ and $y$. An edge from $i$ 
to $j$ in $CG_{A,B}$ indicates that, in some position $p$, $A_p$ contains 
$i$ and $B_p$ contains $j$. This gives $x_i - v_{pi} \geq x_j-v_{pj}$ and $y_j-v_{pj} \geq y_i - v_{pi}$, which implies $x_i - y_i \geq x_j - y_j$. If the edge is directed, then one of these two inequalities is strict, and $x_i - y_i > x_j - y_j$. Intuitively, if we walk from $A$ to $B$, we move more in the $i$ direction than in the $j$ direction. Therefore a cycle in $CG_{A,B}$ would give a series of inequalities which add up to $0 > 0$. 
%

\textit{Surrounding}: Take a point $x$ with type $A$, and an ordered partition $P=(P_1, \ldots, P_r)$. Let $\Delta x = \epsilon (f(1), \ldots, 
f(d))$, where $f(i)$ equals the value of $j$ such that $i\in P_j$. We claim that for $\epsilon$ sufficiently small, $x + \Delta x $ has type $A_P$. Take any 
index $k$. Then the elements of the $k$-th coordinate of the type of $x + \Delta x$ are those $i$ which maximize $(x+\Delta x)_i - v_{ki}$. Since 
$\epsilon$ is sufficiently small, the only way this can happen is if $x_i - v_{ki}$ is maximized for this $i$ (\emph{i.e.} $i\in A_k$). Among these 
elements, $(x+ \Delta x)_i - v_{ki}$ is maximized if and only if $\Delta x_i = \epsilon f(i)$ is maximized, \emph{i.e.} if $i$ is maximal among $A_k$ with 
respect to $P$. This completes the proof.
\end{proof}

\section{Properties of tropical oriented matroids}\label{sec:properties}

The following definitions are motivated by ordinary oriented matroids.


\begin{defn}\label{def:G_A}
Given a type $A$, consider the associated undirected graph $G_A$ with vertex set $[d]$ which is given by connecting $i$ and $j$ if 
there exists some coordinate $A_k$ with $i, j\in A_k$. The \textbf{dimension} of a type $A$ is given by the number of connected components of $A$, minus one. A \textbf{vertex} of a tropical oriented matroid 
is a type $A$ with $G_A$ connected; \emph{i.e.}, one of dimension $0$.  A \textbf{tope} is a type $A = (A_1, \ldots, A_n)$ such that each $A_i$ is a singleton; \emph{i.e.}, one of full dimension $d-1$.
\end{defn}

For tropical hyperplane arrangements, the dimension of a type, as defined above, equals the dimension of the region it describes. \cite[Proposition 17]{DS}

The following lemma will be useful.

\begin{lem}
Refinement is transitive: if $C$ is a refinement of $B$, and $B$ is a refinement of $A$, then $C$ is a refinement 
of $A$.
\end{lem}

\begin{proof}
Suppose $B$ is a refinement of $A$ with respect to the ordered partition $(P_1, \ldots, P_r)$, and $C$ is a 
refinement of $B$ with respect to the ordered partition $(Q_1, \ldots, Q_s)$. Let $X_{ij}=P_i\cap Q_j$. Then it is easy to see that $C$ is the refinement of $A$ given by $(X_{11}, X_{12}, \ldots, X_{1s}, 
X_{21}, \ldots, X_{rs})$.
\end{proof}

\begin{lem}\label{subset-refine}
Suppose that $A$ and $B$ are types of a tropical oriented matroid, and suppose that we have $B_i\subseteq A_i$ for 
all $i\in [n]$. Then $B$ is a refinement of $A$.
\end{lem}

\begin{proof}
Suppose that $B$ is not a refinement of $A$. This means that there is no way to consistently break ties among each 
$A_i$ so that $B_i$ consists of the maximal elements of $A_i$; in other words, the set of equations given by $x_j 
= x_k$ for $j, k\in B_i$ and $x_j > x_k$ for $j\in B_i, k\in A_i\setminus B_i$ has no solution. By linear 
programming duality, this implies that some linear combination of these adds up to $0 > 0$. The inequalities which 
contribute will then form a directed cycle in the comparability graph $CG_{A,B}$, violating the comparability axiom.
\end{proof}

As in ordinary oriented matroids, we have the following theorem.

\begin{thm}\label{topes-determine}
The topes of a tropical oriented matroid $M$ completely determine it. To be precise, $A = (A_1, \ldots, A_n)$ is in 
$M$ if and only if the following two conditions hold:
\begin{itemize}
\item $A$ satisfies the compatibility axiom with every tope of $M$ (\emph{i.e.} $CG_{A, T}$ is acyclic for every tope $T$ 
of $M$.)
\item All of $A$'s total refinements are topes of $M$.
\end{itemize}
\end{thm}

\begin{proof}
First, note that if $A$ satisfies the conditions, so does every refinement $B$ of $A$: the total refinements of a 
refinement of $A$ are a subset of the total refinements of $A$ itself, and the comparability graph $CG_{B, T}$ is a 
subgraph of the comparability graph $CG_{A, T}$ for each $T$.

Suppose that we have a minimal counterexample with respect to refinement: an $n$-tuple $A$ such that every refinement of $A$ is in 
$M$, but $A$ itself is not. 
Throughout the following argument, it will be useful to keep in mind the example $A=(12,69,12,67,23,18,345,135)$.

We know that $A$ has some element that is not a singleton; without loss of generality, assume that $A_1$ contains $\{1, 2\}$. 
Consider the connected components of $G_A\setminus 1$, the graph obtained from $G_A$ by deleting vertex $1$ and the edges incident to it. One of these components contains 2; let $S$ consist of 
this subset of $[d]$, without loss of generality $\{2, \ldots, r\}$, and let $T$ be $\{r+1, \ldots, d\}$. In the example, $S=\{2,3,4,5\}$ and $T=\{6,7,8,9\}$. 

Now consider the refinements 
$B$ given by the partition $(S, \{1\}\cup T)$, and $C$ given by the partition $(\{1\}\cup T, S)$. In our example $B=(1,69,1,67,23,18,345,1)$ and $C=(2,69,2,67,23,18,345,35)$. Since $B_1,C_1 \neq A_1$, we know that $B$ and $C$ are 
proper refinements of $A$, and therefore are in $M$ by the minimality assumption on $A$.

Now eliminate in position 1 between $B$ and $C$ to get 
some element $D$ of $M$. We will prove that $D=A$, thereby showing that $A$ is in $M$.
In our example we have $D=(12,69,*,67,23,18,345,*)$. In general we have $D_1 = B_1\cup C_1 = \{1\}\cup \{A_1\setminus \{1\}\} = A_1$. If position $i$ does not involve any element of $S$, or if it involves elements of $S$ and does not involve $1$, then we have $B_i = C_i = A_i$ so $D_i = A_i$.

The remaining case is when $A_i$
contains $1$ as well as some elements of $S$. In this case, we have $B_i = \{1\}$, $C_i = A_i \setminus \{1\}$. By the elimination 
axiom, $D_i$ is equal to $B_i, C_i$, or $B_i\cup C_i$. 
The proof of the theorem is then complete with the following lemma.
\end{proof}

\begin{lem}
In the situation of the remaining case of Theorem~\ref{topes-determine}, where $A_i$ contains $1$ and some elements of $S$, we have that 
$D_i$ contains 1 and some element of $C_i$, and therefore $D_i = A_i$. 
\end{lem}

\begin{proof}
Suppose there is one such $D_i$ which does not contain both 
1 and some element of $C_i$. Let $x$ be an element of $C_i$. Since $x\in S$, there exists a nontrivial path in $G_A$ from $1$ to $x$ 
consisting only of $1$ and elements of $S$. Choose $D_i$, $x$, and the path from $1$ to $x$ so that the length of this path is minimal. In our example, they could be $D_8$, $x=3$, and the path $123$; we would then know that $D=(12,69,12,67,23,18,345,*)$

For notational convenience, assume that the path from $1$ to $x$ is $12\ldots x$. This means that there exist positions in $A$ containing each of $\{1, 2\}, \ldots, \{x-1, x\}$. In other words, there is a subconfiguration of $A$ given by $(12, 23, 34, 
\ldots, (x-1)x, 1x)$, where the last position is $i$. Since this is the shortest path, none of the numbers $\{1,2, \ldots, x\}$ aside from the given ones appear in any of these positions. In our example, we are talking about the first, fifth, and seventh positions.

If we restrict our attention to the numbers $1, \ldots, x$, then $A = (12, 23, \ldots, (x-1)x, 1x)$ in these positions. This
means that $D$ agrees with $A$ in these positions, except possibly the last one, where it is $\{1\}$ or $\{x\}$. 

Suppose it is $\{1\}$, so that $D = (12, 23, \ldots, (x-1)x, 1)$. Consider a refinement of $D$ by the partition $([d]\setminus 
[x], \{1\}, \{2\}, \ldots, \{x\})$. In these positions, it is equal to $(2, 3, \ldots, x, 1)$; we can take a further total 
refinement to get a tope $T$ with these elements in these positions. But this tope is incomparable to $A$; $CG_{A, T}$ contains the 
directed cycle $(1, 2, \ldots, x)$. This is a contradiction.

Similarly, if $D_i = \{x\}$, then $D = (12, \ldots, (x-1)x, x)$. Refining $D$ by $([d]\setminus [x], x, x-1, \ldots, 2, 
1)$ and further refining yields a tope with $(1, 2, \ldots, x-1, x)$ in these positions. This tope is also incomparable to 
$A$, with $CG_{A, T}$ containing the same cycle in the opposite direction. 
\end{proof}

The vertices also determine the tropical oriented matroid.

\begin{thm}\label{thm:vertices}
A tropical oriented matroid is completely determined by its vertices. To be precise, all types are refinements of vertices.
\end{thm}

\begin{proof}
We need to show that if a type $A = (A_1, \ldots, A_n)$ is not a vertex, then there exists some type of which it is a 
refinement. By Lemma~\ref{subset-refine}, we just need to find a type which strictly contains it. To simplify notation, 
we will do a proof by example; this method clearly works in general.

Suppose that $A$ is not a vertex. This means that the graph $G_A$ is disconnected. There are two possible cases:

\textit{Case 1.} All of the numbers appearing in the $A_i$ are in the same connected component, but there is an element 
of $[d]$ which appears in none of them. Suppose $A = (123, 14, 24, 234, 23)$. Eliminate $A$ with $\textbf{5}=(5,5,5,5,5)$ in position 
1. This yields $B = (1235, B_2, B_3, B_4, B_5)$, where each $B_i$ is equal to $A_i$, $\{5\}$, or $A_i\cup \{5\}$. We are 
done unless some $B_i$ is equal to $\{5\}$. In this case, we eliminate $B$ with $A$ in that position to get a type $C$: 
then $C$ contains $A$ in each coordinate where $B$ does, as well as in the position we just eliminated. Continue 
this process, eliminating in each position where the resulting type is equal to $\{5\}$, until we obtain a type which 
contains $A$ in every coordinate. In the last position we eliminated in, this type also contains $\{5\}$, so 
it 
strictly contains $A$. This completes the proof.

\textit{Case 2.} The numbers appearing in the $A_i$ are in two or more connected components of $G_A$. Take the following 
example:
\[
A = (12,46,256,135,34,78,79,9,7)
\]
so that one connected component of $G_A$ is $[6]$. 
In particular, every coordinate of $A$ is either a subset of $[6]$ or of $[7, 9]$. 
We use the symbol $S^*$ to represent a non-empty set containing only elements from $S$.  The graph $G_A$ is shown in Figure \ref{fig:G_A}. In our example, the induced subgraph of $G_A$ with vertices $1,2,\ldots, i$ is connected for all $1 \leq i \leq 6$. We will need this property, which can be accomplished in the general case by suitable relabelling.

\begin{figure}[h]
\begin{center}\includegraphics[height=4cm]{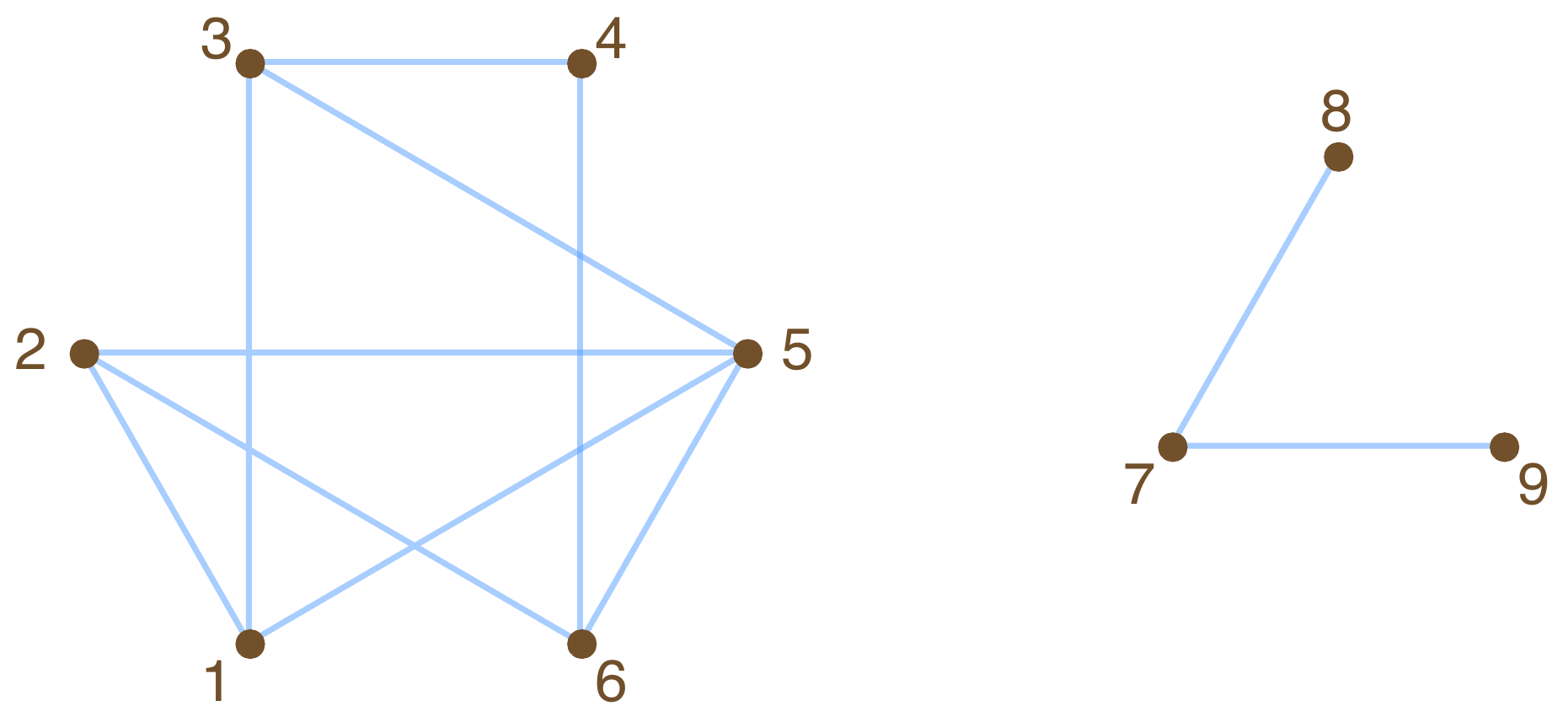}\end{center}
\caption{\label{fig:G_A} 
The graph $G_A$ for $A=(12,46,256,135,34,78,79,9,7)$}
\end{figure}

We claim first that there is a type of the form
\[
A_6=(12,46,256,135,34, [6]^*, [6]^*, [6]^*, [6]^*);
\]
\emph{i.e.}, a type which agrees with $A$ on all coordinates in the $[6]$ category, and consists only of elements from $[6]$. We will build it up one step at a time by building a type $A_i$ ($1\leq i \leq 6$) with elements from $[i]$ such that, wherever the restriction of $A$ to $[i]$ is nonempty, it agrees with $A_i$.

First, eliminate between \textbf{1} and \textbf{2} in position 1. This yields a type 
\[
A_2=(12, [2]^*,[2]^*,[2]^*,[2]^*,[2]^*,[2]^*,[2]^*,[2]^*, [2]^*).
\]
But in any position where $A$ contains a 1 or a 2 (or both), $A_2$  agrees with $A\cap [2]$; if not, then the acyclic comparability graph $CG(A_2,A)$ has a directed edge between elements of $[2]$, which is impossible as it also has an undirected edge between 1 and 2 from the first position. So 
\[
A_2=(12, [2]^*,2,1,[2]^*,[2]^*,[2]^*,[2]^*,[2]^*).
\]
To build $A_3$ 
we eliminate between $A_2$ and \textbf{3} in position 4, to get
\[
A_3=([3]^*,[3]^*,[3]^*,13,[3]^*,[3]^*,[3]^*,[3]^*,[3]^*).
\]
Now, in position 1, $A_3$ equals $3,12$ or $123$; and in fact, it has to equal $12$, since $CG(A_3,A)$ already has an undirected edge between $3$ and $1$. We now see that $\{1,2,3\}$ is connected by undirected edges in $CG(A_3,A)$. Therefore, wherever $A$ contains a 1,2,or 3, $A_3$ must agree with $A\cap [3]$; that is,
\[
A_3=(12,[3]^*,2,13,3,[3]^*,[3]^*,[3]^*,[3]^*).
\]
We continue in this way. To build $A_i$, we eliminate between $A_{i-1}$ and $A$ in a position which connects $i$ to $\{1,\ldots,i-1\}$ in $G_A$. We then ``grow" a spanning tree of $G_A$ restricted to $[i]$, starting at vertex $i$. Each edge in the tree guarantees that $A_i$ and $A \cap [i]$ agree in a new coordinate.  Once we have done this, we know that all of $[i]$ is connected in $CG(A_i,A)$ by undirected edges. Thus, wherever $A$ contains a $1,2,\ldots,$ or $i$, $A_i$ must agree with $A\cap [i]$.

In our example, we obtain $A_4$ by eliminating between $A_3$ and $A$ in position 5. We grow the spanning tree with edges 43, 31, 12 in that order. This guarantees, in that order, that position 5 is 34, position 4 is 13, and position 1 is 12. This means that $[4]$ is  connected in $CG(A_4,A)$ by undirected edges, which forces position 2 to be 4 and position 3 to be 2. Therefore
\[
A_4 = (12,4,2,13,34,[4]^*,[4]^*,[4]^*,[4]^*).
\]
We eliminate again with \textbf{5} in position 3, and then with \textbf{6} in position 3, to obtain the desired type.

Now that we have a type 
\[
A_6=(12,46,256,135,34, [6]^*, [6]^*, [6]^*, [6]^*),
\]
we proceed as in case 1. We eliminate $A_6$ with $A$ in position 6 to obtain an element which contains $A$ in 
the first seven positions, and 
strictly contains $A$ (in position 6.) If there are any positions where the new type does not contain $A$, it 
consists there of a subset of 
$[6]^*$; we eliminate in that position with $A$. We continue doing this, eventually obtaining an element which contains $A$ in every position and strictly contains 
$A$ in the position where we last eliminated.
\end{proof}

The next two propositions establish the tropical analogues of the standard matroid operations of deletion and contraction.

\begin{prop}
Let $M$ be a tropical oriented matroid with parameters $(n, d)$. Pick any coordinate $i\in [n]$. Then the \textbf{deletion} $M_{\setminus 
i}$, which consists of all $(n-1, d)$ types which arise from types of $M$ by deleting coordinate $i$, is also a tropical oriented matroid.
\end{prop}

\begin{proof}
It is straightforward to verify each axiom. For boundary, the deletion of $\textbf{j}=(j, j, \ldots, j)$ is again $(j, j, \ldots, j)$. To 
eliminate between two types of $M_{\setminus i}$, simply find any preimages of them in $M$ and eliminate between these. The 
comparability graph $CG_{A, B}$ is a subgraph of the comparability graph of any preimages of $A$ and $B$ 
in $M$, and hence is 
acyclic. Finally, the surrounding axiom holds, as refinement commutes with deletion.
\end{proof}

\begin{prop}
Let $M$ be a tropical oriented matroid with parameters $(n, d)$. Pick any direction $i\in [d]$. Then the \textbf{contraction} $M_{/i}$, which 
consists of all types of $M$ which do not contain $i$ in any coordinate, is also a tropical oriented matroid (with parameters $(n, d-1)$.)
\end{prop}

\begin{proof}
Again, it is straightforward to verify each axiom; assume that $i=d$ for notational convenience. For $j\in [d-1]$, \textbf{j} is a type of $M$ 
not containing $d$, and hence is a type of $M_{/d}$, so the boundary axiom holds. When we eliminate between two types, no new symbols are 
introduced, so the elimination of two types of $M_{/d}$ in $M$ is again a type of $M$ not containing $d$, and hence is in $M_{/d}$. The comparability graph of two types in $M_{/d}$ is the same as their comparability graph in $M$, except for the isolated vertex $d$ (which is removed), and is therefore also acyclic. The surrounding axiom is trivial, as the process of refinement does not introduce any new symbols, and hence any refinement in $M$ of a type in $M_{/d}$ is again in $M_{/d}$. 
\end{proof}
\section{Three conjectures}
\label{sec:conjectures}

In this section, we explore one of the motivations for the study of tropical oriented matroids: their connection to triangulations of products of simplices. 
Realizable tropical oriented matroids (\emph{i.e.} tropical hyperplane arrangements) have a canonical bijection to regular subdivisions of products of 
simplices~\cite{DS}. Our main conjecture is the following:

\begin{conj}\label{anytriang}
There is a one-to-one correspondence between the subdivisions of the product of simplices $\Delta_{n-1} \times \Delta_{d-1}$ and the tropical oriented matroids with parameters $(n,d)$.
\end{conj}

In Section \ref{sec:2d} we will prove the backward direction, and the forward direction for triangulations of $\Delta_{n-1} \times \Delta_2$. Thus the conjecture is reduced to proving that subdivisions of $\Delta_{n-1} \times \Delta_{d-1}$ satisfy the elimination axiom.

The conjectural correspondence is as follows: 
Give the vertices of $\Delta_{n-1} \times \Delta_{d-1}$ the labels $(i, j)$ for $1 \leq i \leq n$ and $1 \leq j \leq d$. Given a triangulation $T$ of $\Delta_{n-1} \times \Delta_{d-1}$, we define the \textbf{type} of a face $F$ of $T$ to equal $(S_1, \ldots, S_n)$, where $S_i$ consists of those $j$ for which $(i, j)$ is a vertex of $F$. Consider the types of the faces which contain at least one vertex from each of the $n$ copies of $\Delta_{d-1}$; \emph{i.e.}, those whose types have no empty coordinates.
We conjecture that this is the collection of types of a tropical oriented matroid and, conversely, that every tropical oriented matroid arises in this way from a unique subdivision.

Consider, for example, the triangulation of the prism $\Delta_1 \times \Delta_2 = 12 \times 123$ shown in the left panel of Figure \ref{fig:cayley.trick}; it consists of the three tetrahedra $\{(1,1), (1,2), (1,3), (2,1)\},$ $\{(1,2), (1,3), (2,1), (2,3)\}$ and $\{(1,2), (2,1), (2,2), (2,3)\}$.
These tetrahedra have types $(123,1), (23,13),$ and $(2,123)$, and these are the vertices of a tropical oriented matroid with parameters $(2,3)$. Higher dimensional types of the tropical oriented matroid correspond to lower dimensional faces of the triangulation.

\medskip

A proof of Conjecture \ref{anytriang} would give us a form of duality for tropical oriented matroids. This duality exists for tropical hyperplane arrangements~\cite{DS}, since regular subdivisions of $\Delta_{n-1} \times \Delta_{d-1} \cong \Delta_{d-1}\times \Delta_{n-1}$ are in canonical bijection with $(n, d)$-hyperplane arrangements and $(d, 
n)$-hyperplane arrangements. One can then guess how this should extend to tropical oriented 
matroids in general.

\begin{defn}
A \textbf{semitype} (with parameters $(n, d)$) is given by an $n$-tuple of subsets 
of $[d]$, not necessarily nonempty. Given a tropical oriented matroid $M$, its \textbf{completion} $\widetilde{M}$ consists 
of all semitypes which result from types of $M$ by changing some subset of the 
coordinates to the empty set. Given a collection of semitypes, its \textbf{reduction} consists of all honest types contained in the 
collection. 
\end{defn}

\begin{defn}
Let $A$ be a semitype with parameters $(n, d)$. Then the \textbf{transpose} $A^T$ of $A$, a semitype with parameters $(d, n)$ (\emph{i.e.} a 
$d$-tuple of subsets of $[n]$), has $i\in A^T_j$ 
whenever $j\in A_i$. 
\end{defn}

Essentially, a type can be thought of as a 0-1 $n\times d$ matrix (or alternatively a bipartite graph with $n$ left vertices and $d$ right vertices), which can be interpreted as either an $n$-tuple of subsets of $[d]$, or a $d$-tuple of subsets of $[n]$. The transpose operation is simply the obvious map between the two.

\begin{defn}
Let $M$ be a tropical oriented matroid. Then the \textbf{dual} of $M$ is the reduction of the collection of 
semitypes given by transposes of semitypes in $\widetilde{M}$. 
\end{defn}

In other words, the dual is just the reinterpretation of the types of $M$ as $d$-tuples of subsets of $n$ instead of the other way around; the 
restriction and completion operations are required for purely technical reasons and are inessential to the intuition. By definition, if the dual $M^*$ of $M$ is indeed a tropical matroid, then clearly $M^{**}=M$.

As previously mentioned 
the dual of a realizable tropical oriented matroid is again a realizable tropical oriented matroid, and Conjecture \ref{anytriang} would show that this operation works in general.

\begin{conj}
The dual of a tropical oriented matroid with parameters $(n, d)$ is a tropical oriented matroid with parameters $(d, 
n)$.
\end{conj}

Note that as in ordinary matroid theory, the operations of deletion and contraction are dual to each other.

\medskip

A proof of Conjecture \ref{anytriang} would also give us a topological representation theorem for tropical oriented matroids, as follows. 

\begin{defn}
A \textbf{tropical pseudohyperplane} is a subset of $\TP^{d-1}$ which is PL-homeomorphic to a tropical hyperplane.
\end{defn}

\begin{conj}\label{toprep}\emph{(Topological representation theorem.)}
Every tropical oriented matroid can be realized by an arrangement of tropical pseudohyperplanes.
\end{conj}

Let us sketch the idea of a proof of Conjecture \ref{toprep} assuming that Conjecture \ref{anytriang} is true.
The first step is to apply the Cayley trick to biject a triangulation of a product of simplices $\Delta_{n-1}\times \Delta_{d-1}$ to a mixed 
subdivision of the dilated simplex $n\Delta_{d-1}$. 

\begin{figure}[h]
\centering
\includegraphics[height=4cm]{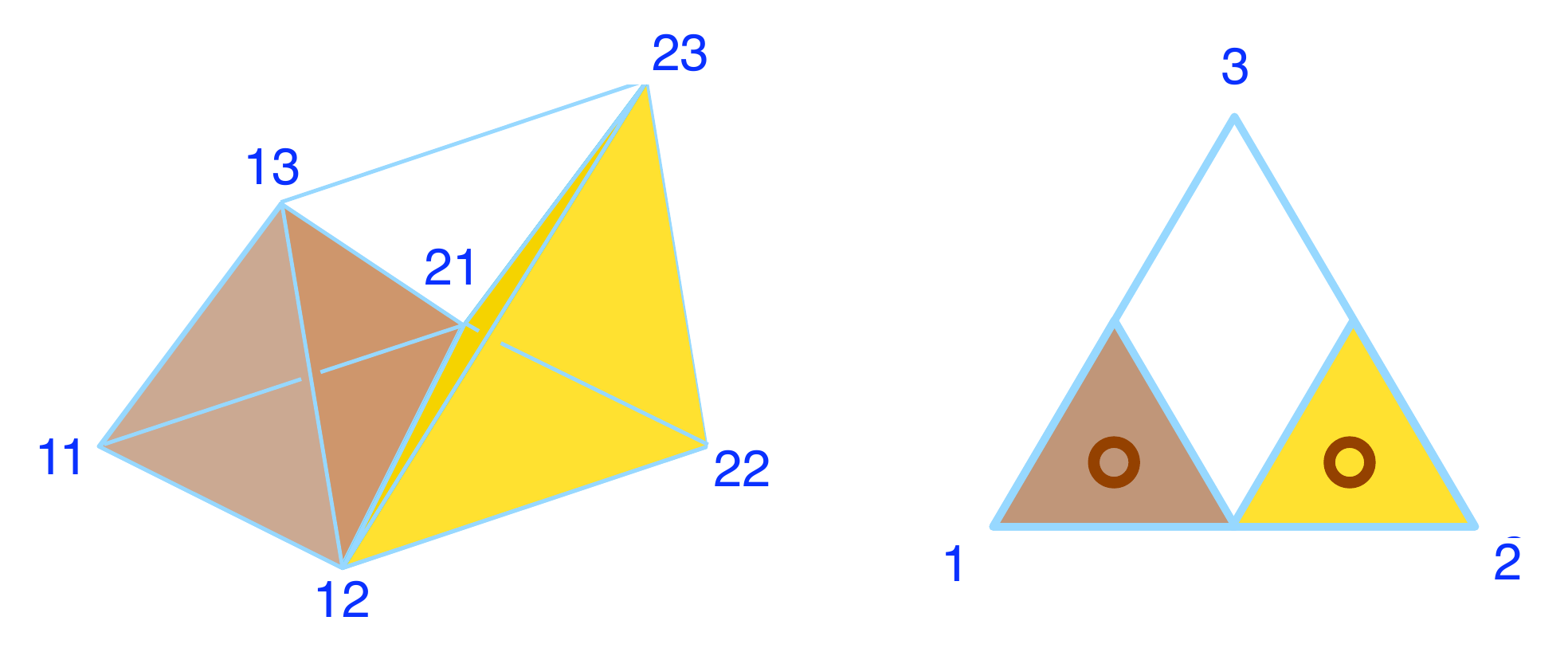}
\caption{The Cayley trick.} \label{fig:cayley.trick}
\end{figure}

This process is detailed in~\cite{Santos}; it is a standard trick in polyhedral geometry. Let us briefly illustrate this correspondence for
the tiling of $\Delta_1 \times \Delta_2$ shown in Figure \ref{fig:cayley.trick}: the tetrahedra have types 
$(123,1), (23,13),$ and $(2,123)$, so they get mapped to the Minkowski sums $123+1$, $23+13$, and $2+123$. These pieces form a mixed subdivision of the triangle $2\Delta_2$, as shown.


The second step is to consider the \textbf{mixed Voronoi subdivision} of this mixed subdivision of $n \Delta_{d-1}$, defined as follows. The \textbf{Voronoi subdivision} of a $k$-simplex divides it into $k$ regions, where region $i$ consists of the points in the simplex for which $i$ is the closest vertex. 
We subdivide each cell $S_1+S_2+\cdots+S_n$ in our mixed subdivision into the regions $R_1+R_2+\cdots+R_n$, where $R_i$ is a region in the Voronoi subdivision of $S_i$. For example, in two dimensions the finest mixed cells we can get are a triangle and a rhombus, and their mixed Voronoi subdivisions are shown in Figure \ref{rhombi}.

The lower-dimensional faces introduced by this mixed Voronoi subdivision will all fit together to form a 
tropical pseudohyperplane arrangement; each simplex cell in the mixed subdivision corresponds to one apex of a tropical pseudohyperplane, and the 
other cells dictate how these pseudohyperplanes propagate throughout the diagram. 
This process is shown in Figure~\ref{voronoi} for a mixed subdivision of $4 \Delta_2$.\footnote{Note that tropical lines are usually drawn with angles of $90^{\circ}, 135^{\circ}, 135^{\circ}$, while here they appear, more symmetrically, with three angles of $120^{\circ}$.} 
To recover the collection of types, since each 
of the $n$ tropical pseudohyperplanes divides the figure into $d$ canonically indexed sectors, one can 
simply take the types of all regions (of all dimensions) in the tropical pseudohyperplane arrangement. These are precisely the types of the triangulation we started with.

\begin{figure}[h]
\begin{center}\includegraphics[height=4cm]{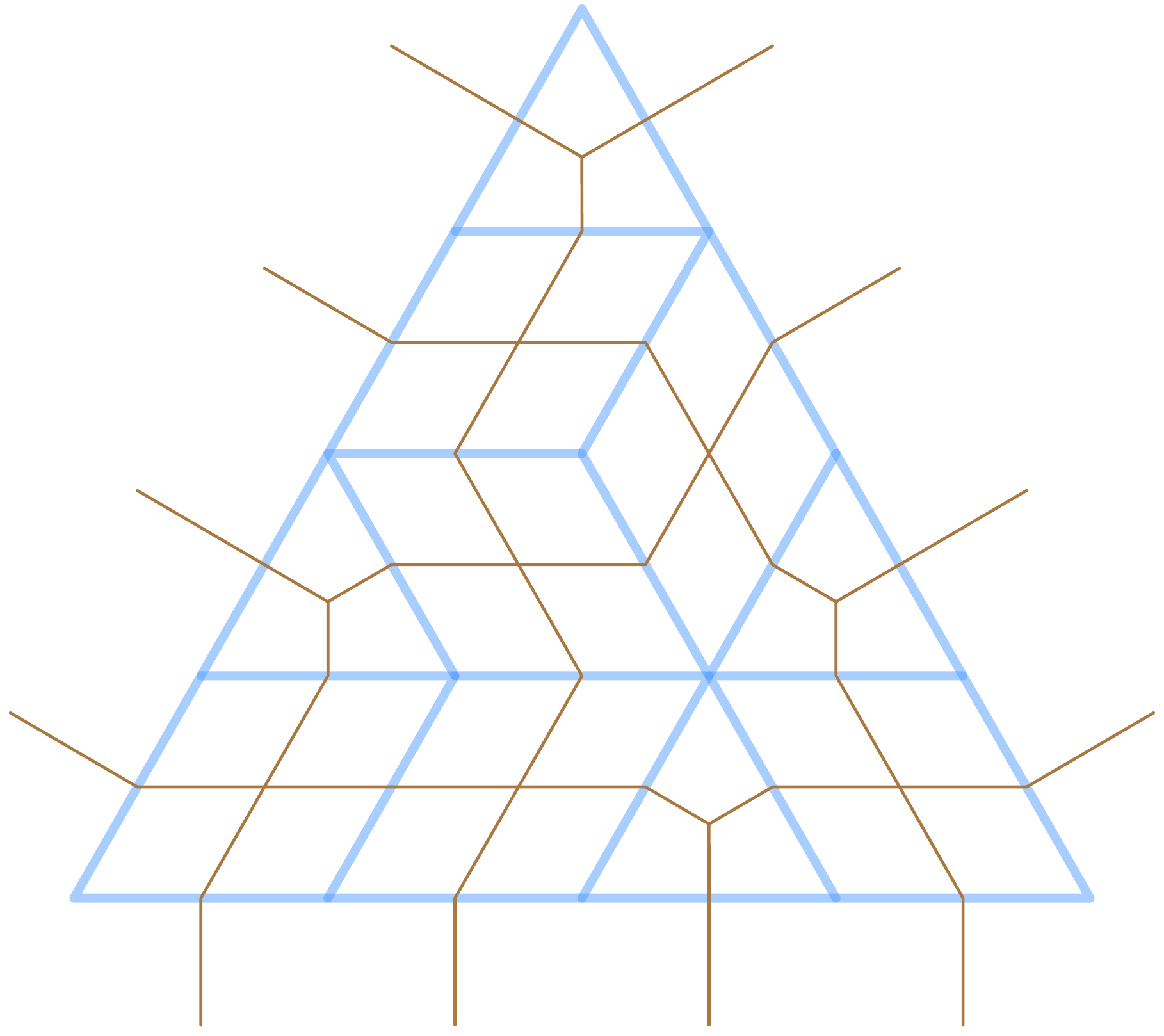}\end{center}
\caption{\label{voronoi} How to obtain a tropical pseudohyperplane arrangement from a Cayley trick picture of a triangulation (here of $\Delta_3\times\Delta_2$.)}
\end{figure}

For an example in three dimensions, the top panel of Figure~\ref{pseudohyps-3d} shows a mixed subdivision of $2 \Delta_3$ (which corresponds to a triangulation of $\Delta_1 \times \Delta_3$)  and the Voronoi mixed subdivision of each one of its four cells.
The bottom panel shows how the lower-dimensional faces introduced by this mixed Voronoi subdivision fit together to form two tropical pseudohyperplanes, shown in different colors. As one should expect, these two pseudohyperplanes intersect in a tropical pseudoline, which is dotted in the diagram.

\begin{figure}[h]
\begin{center}\includegraphics[height=6cm]{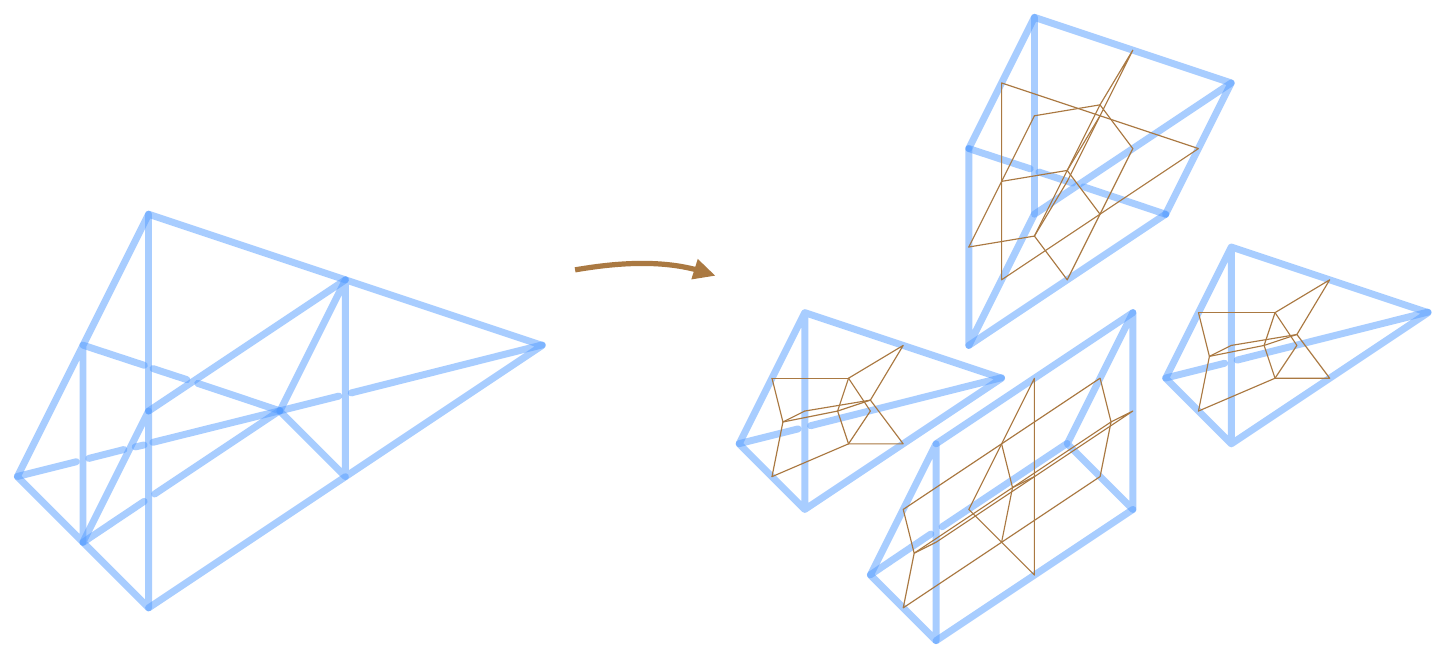}\end{center}
\bigskip
\bigskip
\begin{center}\includegraphics[height=8cm]{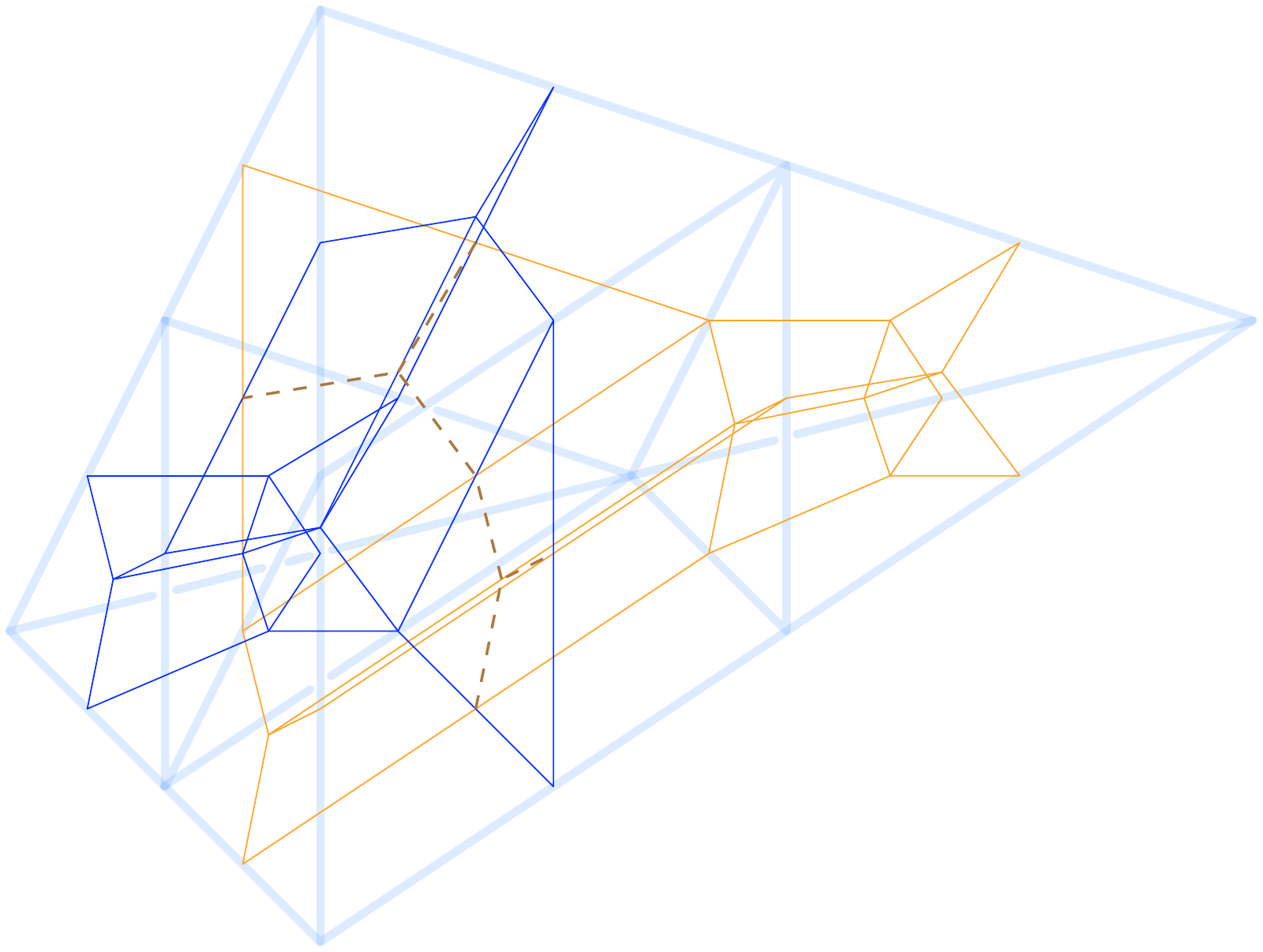}\end{center}
\caption{\label{pseudohyps-3d} A mixed subdivision of $2 \Delta_3$, the Voronoi subdivision of each cell, and the resulting tropical pseudohyperplane arrangement in $\TP^3$.}
\end{figure}

\medskip

We conclude this section with two possible applications of these ideas.

Firstly, one should be able to use the topological representation theorem for tropical oriented matroids to 
prove that the triangulations of $\Delta_{n-1} \times \Delta_{d-1}$ are flip-connected. It is a natural 
question to ask whether the triangulations of a polytope $P$ can all be reached from one another by a series 
of certain local moves, known as flips. This is not true in general \cite{Santosflips}, and there are only a 
few polytopes which are known to be flip-connected, including convex polygons, cyclic polytopes \cite{R}, and 
products $\Delta_{n-1} \times \Delta_2$. \cite{Santos} In contrast, the analogous statement is true for 
regular triangulations, due to the fact that there is a continuous model for them.  If Conjecture 
\ref{toprep} 
is true, then tropical pseudohyperplane arrangements constitute a continuous model for the triangulations of 
$\Delta_{n-1} \times \Delta_{d-1}$, and moving around the parameter space of tropical pseudohyperplane 
arrangements should give a proof of their flip connectivity.

Secondly, tropical oriented matroids may give a proof of a conjecture describing the possible locations of 
the simplices in a fine mixed subdivision of $n \Delta_{d-1}$. In studying the Schubert calculus of the flag 
manifold, the first author and Sara Billey \cite{AB} described the matroid ${\mathcal T}_{n,d}$ of the line 
arrangement 
determined by intersecting $d$ generic flags in $\RR^n$. They showed that this matroid is closely related to 
the fine mixed subdivisions of $n \Delta_{d-1}$: every such subdivision has exactly $n$ simplices, which are 
a basis of the matroid ${\mathcal T}_{n,d}$. In the converse direction, they conjectured that every basis 
comes from such a subdivision, and they proved it for $d=3$. To prove this statement, one would need a good 
way of constructing subdivisions. This may be approached by taking advantage of the continuous model of 
tropical pseudohyperplane arrangements, or by developing a toolkit for building tropical oriented matroids, 
in analogy with the multiple constructions available in ordinary matroid theory.

\section{Tropical oriented matroids and subdivisions of $\Delta_{n-1} \times \Delta_{d-1}$.}\label{sec:2d}

In this section we make progress towards Conjecture \ref{anytriang}, which relates tropical oriented matroids and subdivisions of $\Delta_{n-1} \times \Delta_{d-1}$. We prove one direction of the conjecture for all $n$ and $d$, and the other direction in the special case of triangulations of $\Delta_{n-1} \times \Delta_2$. 

To do so, let us review a combinatorial characterization of these subdivisions.
Each vertex of $\Delta_{n-1} \times \Delta_{d-1}$ corresponds to an edge of the bipartite graph
$K_{n,d}$. The vertices of each subpolytope in $\Delta_{n-1} \times
\Delta_{d-1}$ determine a subgraph of $K_{n,d}$. Each
subdivision of $\Delta_{n-1} \times \Delta_{d-1}$ is then
encoded by a collection of subgraphs of $K_{n,d}$. Figure
\ref{fig:trees} shows the three trees that encode the
triangulation of Figure \ref{fig:cayley.trick}.

\begin{figure}[h]
\centering
\includegraphics[height=2.5cm]{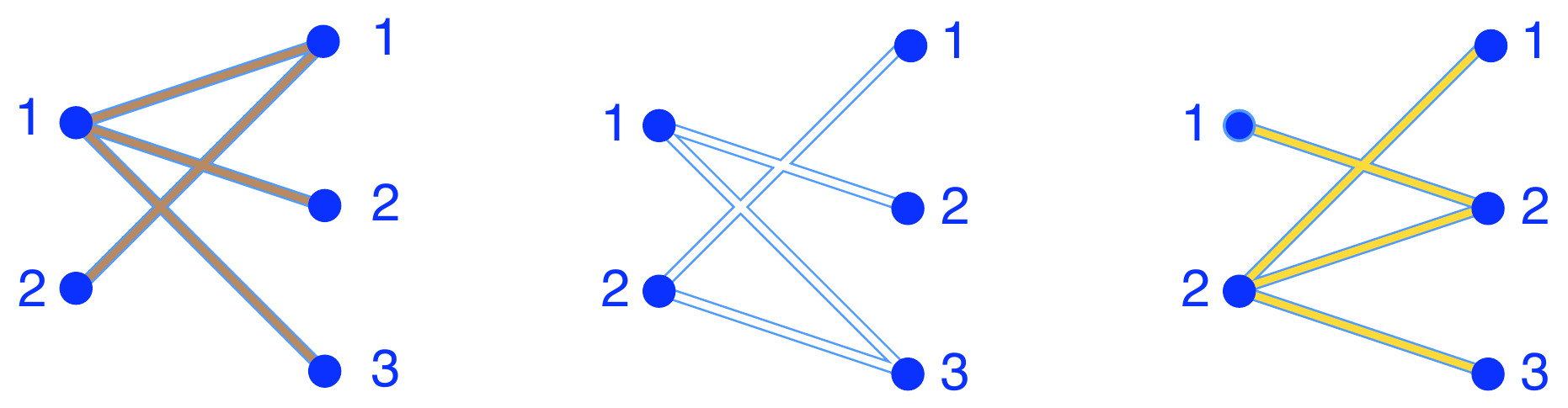}
\caption{The trees corresponding to the triangulation of Figure
\ref{fig:cayley.trick}.} \label{fig:trees}
\end{figure}

\begin{thm}\cite{AB, Santostriangs}\label{prop:trees}
A collection of subgraphs $t_1, \ldots, t_k$ of $K_{n,d}$ encodes
a subdivision of $\Delta_{n-1} \times \Delta_{d-1}$ if and only
if:
\begin{enumerate}
\item Each $t_i$ spans $K_{n,d}$.

\item For each $t_i$ and each maximal disconnected subgraph $s_i$ of $t_i$,
either $s_i$ has an isolated vertex or there is another $t_j$ containing $s_i$. 
%

\item If there are two subgraphs $t_i$ and $t_j$ and a cycle
$C$ of $K_{n,d}$ which alternates between edges of $t_i$ and edges
of $t_j$, then both $t_i$ and $t_j$ contain $C$.
\end{enumerate}
\end{thm}

For triangulations, in (1) we need each $t_i$ to be a spanning tree, in (2) the subgraph $s_i$ can be $t_i-e$ for any edge $e$ of $t_i$, and in (3) one cannot have a cycle alternating between $t_i$ and $t_j$. 

Intuitively, (1) guarantees that the pieces of the subdivision are full-dimensional; (2) says that if we walk out of one of the pieces through one of its facets, we will either walk out of the polytope or into another piece of the subdivision; (3) guarantees that the pieces intersect face-to-face.

\begin{thm}
The types of the vertices of a tropical oriented matroid $M$ with parameters $(n,d)$ describe a subdivision of $\Delta_{n-1} \times \Delta_{d-1}$.
\end{thm}

\begin{proof}
Let $t_A$ be the subgraph of $K_{n,d}$ determined by a type $A$ of $M$; this graph is related to the subgraph 
$G_A$ of $[d]$ of Definition \ref{def:G_A} as follows: vertices $i$ and $j$ are connected in the graph $G_A$ 
of a type $A$ if and only if, in the bipartite graph $t_A$, there is a vertex in $[n]$ connected by an edge 
to both $i$ and $j$ in $[d]$. Now we check the conditions of Theorem \ref{prop:trees}. 

(1) Since a type $A$ has no empty coordinates, it is clear that $G_A$ is connected if and only if $t_A$ spans $K_{n,d}$.

(2) Let $t_i$ and $s_i$ correspond to types $A$ and $B$. Since $s_i$ is maximal disconnected, $G_B$ has exactly two connected components. If one of them is an isolated vertex, we are done. Otherwise, we are in Case 2 of the proof of Theorem \ref{thm:vertices}. If, for instance, $B=(12,46,256,135,34,78,79,7)$, then that proof constructs a type $C=(12,46,256,135,34,*,*,*)$ containing $B$ which is identical to $B$ in the coordinates corresponding to the first component $[6]$, and adds some elements of $[6]^*$ to the coordinates corresponding to the second component $\{7,8,9\}$. This makes $G_C$ connected, so $C$ is a vertex. In the same way, we could have constructed a vertex $C'$ containing $B$ of the form $C'=(*,*,*,*,*,78,79,7)$. These types $C$ and $C'$ are distinct, so one of them gives us the second subgraph $t_j$ containing $s_i$. 

(3) A cycle $C$ of $K_{n,d}$ which alternates between edges of $t_A$ and $t_B$ would give rise to a cycle in the comparability graph $CG(A,B)$ involving the $[d]$-vertices of $C$.
\end{proof}

\begin{thm}
The types of the full-dimensional simplices of a (possibly nonregular) triangulation of $\Delta_{n-1} \times \Delta_2$ are the vertices of a tropical oriented matroid.
\end{thm}

\begin{proof}
Consider the collection of types given by taking all cells which contain at least one element from each of the $n$ copies of $\Delta_2$. 
We will show that this is the collection of types of a tropical oriented matroid, and the result will follow.

Most of the axioms are straightforward to verify for any subdivision of $\Delta_{n-1} \times \Delta_{d-1}$. The boundary axiom is easy, as for each $i\in [d]$, $\{(1, i), \ldots, (n, i)\}$ is a face: it is one of 
the $n$ copies of $\Delta_{d-1}$. If two cells $A,B$ of the triangulation violate incomparability, then their graphs $t_A$ and $t_B$ will overlap on a cycle.
The surrounding axiom is also easy; this just involves taking the face of a cell given by a 
linear functional whose coefficients are ordered as in the ordered partition.

The elimination axiom is the substantial one. 
We saw that, for tropical hyperplane arrangements, eliminating between types $A$ and $B$ at position $i$ amounts to walking along a tropical line from $A$ to $B$, and finding its largest intersection with hyperplane $i$. To mimic that proof, we need to understand how to walk around the triangulation.

One can understand this in at least two ways: using the Voronoi picture or Theorem \ref{prop:trees}. Let us first describe how to walk around the mixed subdivision corresponding to the given triangulation. The possible puzzle pieces in a mixed subdivision of the Minkowski sum of $n \Delta_2$ are 
$A=123$, $B = 13+23$, $C=12+23$, and $D=12+13$, as shown in Figure \ref{rhombi}. 

\begin{figure}[h]
\begin{center}\includegraphics[height=3cm]{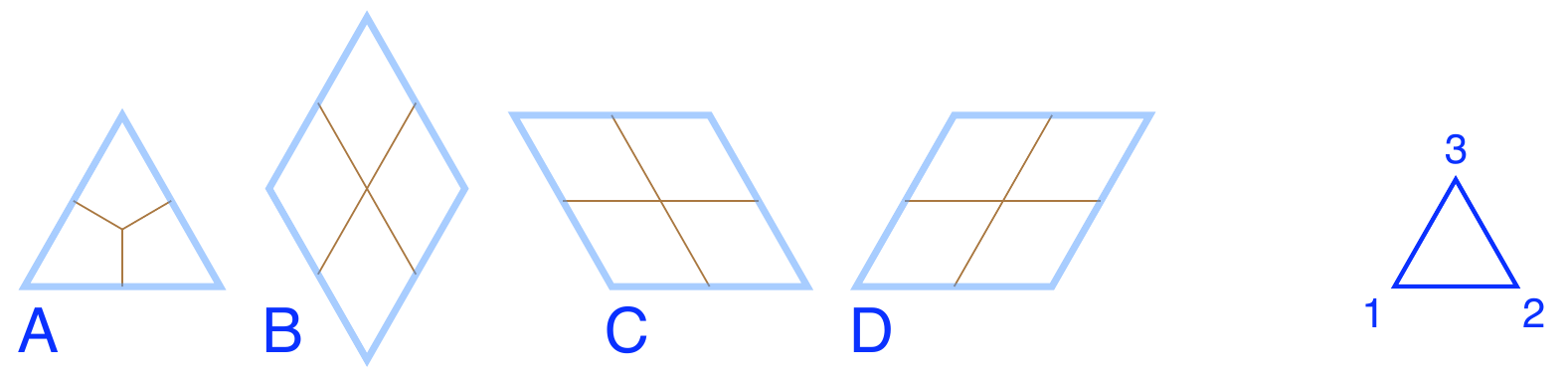}\end{center}
\caption{\label{rhombi} The puzzle pieces in a Voronoi picture corresponding to a mixed subdivision, and the triangle $\Delta_2$.}
\end{figure}

Armed with these Voronoi cells, we can say in great detail what the neighbors of a given vertex look like. For instance, consider a vertex coming 
from a triangle (a type $A$ vertex.) This vertex has type $(123, *, \ldots, *)$, where each $*$ is a singleton. If we exit the triangle on the 
bottom side along 
the tropical hyperplane, then we either encounter the 12-edge of $n\Delta_d$ (if and only if each $*$ is 1 or 2), or we run into puzzle piece $C$ or $D$. If the 
adjacent piece is $C$, then following the ray we are in, some index 3 in some other coordinate will be changed to 23 when we reach the vertex at the 
center of the puzzle piece. Similarly, if puzzle piece $D$ is adjacent to our $A$ piece along the 12-edge, then some index 3 will be changed to 13. So 
we have the following transition rule:
\[
(12{\bf 3}, 3) \rightarrow (12, 3x);
\]
\emph{i.e.}, whenever we have a type with 123 in it, and a 3 anywhere else, then we can find an adjacent vertex which loses the boldface ${\bf 3}$ and changes some 3 to either 31 or 32.  
(Note that this is only the case for one 3-singleton coordinate, and we do not know which one.)

\medskip

Let us now describe the same transition rule from the point of view of Theorem \ref{prop:trees}. If we are in a cell $C_i$ described by tree $t_i$ and wish to walk to a neighboring cell, we first choose the facet of $C_i$ that we wish to cross, which is described by the forest $t_i-e$ for some edge $e$ of $t_i$. If
 edge $e$ is a leaf of $t_i$, that facet is on the boundary of the triangle; if we cross it, we will walk out of the triangle. If, instead, $e$ is an internal edge, then we will cross to some cell $C_j$ which must be described by a tree of the form $(t_i-e) \cup f$. The edge $f$ must join the two connected components of $t_i-e$. This is precisely what condition (3) of Theorem \ref{prop:trees} describes.

Now imagine that we have a triangle with type $(123,*,\ldots,*)$ where each $*$ is a singleton, and we wish 
to exit the triangle along the bottom edge. This amounts to removing the $3$ from the first coordinate. In the corresponding tree, that disconnects vertex $3$ from vertices $1$ and $2$. To reconnect them, one must add an edge from another vertex, which is already connected to $3$, to one of the vertices $1$ and $2$. Once again, this transition is:
\[
(12{\bf 3}, 3) \rightarrow (12, 3x).
\]

We can use similar logic to obtain all of the transition rules. 
Here are the rules for type $A$ and type $C$ vertices (type $B$ and $D$ vertices are 
isomorphic to type $C$ vertices upon permutation of the numerals 1, 2, 3):

\begin{eqnarray*}
({\bf 1}23, 1) & \rightarrow & (23, 1x) \\
(1{\bf 2}3, 2) & \rightarrow & (13, 2x) \\
(12{\bf 3}, 3) & \rightarrow & (12, 3x) \\
\end{eqnarray*}
\begin{eqnarray*}
({\bf 1}2, 23, 1) & \rightarrow & (2, 23, 1x) \\
(1{\bf 2}, 23, x) & \rightarrow & (1, 23, x1) \textrm{ or } (1, 123, x)\\
(12, {\bf 2}3, x) & \rightarrow & (12, 3, x3) \textrm{ or } (123, 3, x)\\
(12, 2{\bf 3}, 3) & \rightarrow & (12, 2, 3x) \\
\end{eqnarray*}

So for a triangle, the generalized rule is that 
we remove some numeral from 123 and then add something to another appearance of that numeral (intuitively, ``moving away from 1" to change 
$({\bf 1}23, 1)$ to $(23, 1x)$). For a rhombus, we either remove a non-duplicated numeral and add something to another appearance of that numeral (``moving 
away from 3" to change $(12, 2{\bf 3}, 3)$ to $(12, 2, 3x)$), or remove a duplicated numeral and add the resulting singleton somewhere else (``moving 
towards 3" to change $(12, {\bf 2}3, x)$ to $(12, 3, x3)$ or $(123, 3, x)$.) All of this agrees with the intuition from ordinary tropical hyperplane 
arrangements, and just corresponds to changing the actual coordinates in the manner indicated (or, in the regular triangulations picture, modifying 
the coordinates of the face-defining hyperplane.) 

\medskip

Armed with these transition rules, which encode the ways one can move around a triangulation of $\Delta_{n-1} \times \Delta_2$, we are ready to prove the elimination axiom. 
Given two types $A = (A_1, \ldots, A_n)$ and $B = (B_1, \ldots, B_n)$, and a position $i$, we need to 
find a type which is equal to $A_i\cup B_i$ in the $i$-th coordinate and either $A_i, B_i$, or $A_i\cup B_i$ in every other coordinate. Since we 
have a triangulation, any subset of a type is a type, which means that we can reduce to the situation where both $A$ and $B$ are maximal among the 
purported collection of types (\emph{i.e.} are full-dimensional cells of the triangulation.) All such types either have one tripleton $123$ and all 
singletons otherwise, or have two non-identical doubletons and all singletons otherwise.

Our proof of this is by induction on the number of positions $j$ where $A_j$ and $B_j$ do not satisfy a containment relation; call this number the 
non-containment index of the pair $(A, B)$. Obviously if this 
number is zero, we can take whichever of $A$ and $B$ has larger $i$-th coordinate, and this will satisfy the requirements of the elimination axiom 
(since it is \textit{a fortiori} equal to either $A$ or $B$ on all other coordinates.) If this number is not zero, we will modify either $A$ or $B$ 
slightly to descend. In other words, suppose we modify $A$; we will produce a type $A'$ for which $A'$ and $B$ satisfy more containment relations than 
$A'$ and $B$, and for which each of $\{A'_j, B_j, A'_j\cup B_j\}$ contains one of $\{A_j, B_j, A_j\cup B_j\}$ for every $j$, with $A'_i\cup B_i 
\supseteq A_i\cup B_i$. We then eliminate between $A'$ and $B$ by the inductive hypothesis, and if necessary remove extra elements 
to satisfy the elimination axiom between $A$ and $B$.

We will carry out this plan by a case-by-case analysis of the possibilities for $A$ and $B$.
For each case, we will describe the types $A$ and $B$ in the top and bottom rows of a matrix, listing the coordinates of both 
types in the positions where either has a non-singleton; there are an arbitrary number of other coordinates $j$, each of which has both $A_j$ and $B_j$ 
as a singleton. We will list the possible pairs for these singletons as well; all other pairs will be ruled out due to incomparability. 
In each case we modify either $A$ and $B$ by using the transition rules to remove one numeral, indicated in boldface, from one of the coordinates. 

The following cases will cover all possibilities up to isomorphism. We will explain every case that involves a new idea; every other case is essentially identical to one of the ones preceding it.

\textbf{Case 1.} There is some position in which both types are doubletons. In this case, it is easy to check that there is exactly one such position.

\textit{Case 1a.} 
$\begin{pmatrix}
12 & 23 & 2 \\
12 & 3 & 2{\bf 3} 
\end{pmatrix}$
Possible singletons: 
$\begin{pmatrix}
1 & 2 & 3 & 1 & 2 \\
1 & 2 & 3 & 3 & 3
\end{pmatrix}$

If $i=3$, then we are already done. Otherwise, we can apply the rule to $B$ sending $(12, 2{\bf 3}, 3)$ to 
$(12, 2, 3x)$. This gets rid of the boldface {\bf 3}, 
and adds $x$ to some other singleton 3; the resulting type $B'$ is unchanged in every other position. If this is in position 2, it must end up as 23 
(otherwise $B'$ is incomparable with $A$), whereupon we must actually have produced $A$ and are done. Otherwise, the only possible cases are that it 
sends some other 3 in position $j>3$ to 13 (whereupon by comparing with $A$, the corresponding position of $A$ must have a 1 in it) or that it 
sends 3 to 23 (similarly, the corresponding position of $A$ must have a 2 in it in this case.) In either of these cases, another containment 
relation is created, as desired by our induction.

\textit{Case 1b.}
$\begin{pmatrix}
12 & 2{\bf 3} & 3 \\
12 & 2 & 13 
\end{pmatrix}$
Possible singletons: 
$\begin{pmatrix}
1 & 2 & 3 & 3 & 3 \\
1 & 2 & 3 & 1 & 2 
\end{pmatrix}$

\textit{Case 1c.}
$\begin{pmatrix}
12 & {\bf 2}3 & 2 \\
13 & 3 & 12 
\end{pmatrix}$
Possible singletons: 
$\begin{pmatrix}
1 & 2 & 3 & 1 & 2 & 2\\
1 & 2 & 3 & 3 & 1 & 3
\end{pmatrix}$

\textit{Case 1d.}
$\begin{pmatrix}
12 & {\bf 2}3 & 2 \\
13 & 3 & 23 
\end{pmatrix}$
Possible singletons: 
$\begin{pmatrix}
1 & 2 & 3 & 1 & 2 & 2\\
1 & 2 & 3 & 3 & 1 & 3
\end{pmatrix}$

\textit{Case 1e.}
$\begin{pmatrix}
12 & {\bf 2}3 & 1 \\
23 & 3 & 12 
\end{pmatrix}$
Possible singletons: 
$\begin{pmatrix}
1 & 2 & 3 & 1 & 1 & 2\\
1 & 2 & 3 & 2 & 3 & 3
\end{pmatrix}$

\smallskip

\textbf{Case 2.} Both types come from rhombi, and there is no overlap among the positions in which they have doubletons.

\textit{Case 2a.}
$\begin{pmatrix}
12 & 13 & 2 & 2\\
1 & 1 & {\bf 1}2 & 13 
\end{pmatrix}$
Possible singletons: 
$\begin{pmatrix}
1 & 2 & 3 & 2 & 2 & 3\\
1 & 2 & 3 & 1 & 3 & 1
\end{pmatrix}$

As before, if $i=3$, we are done. Otherwise, we remove the boldface {\bf 1} using the rule $({\bf 1}2,13,x) \rightarrow ((2,13,x2) \textrm{ or } (2,123,x))$, which changes some $x$ to $x2$ (or 13 to 123). If this is position 1, the non-containment 
index remains constant, but we have reduced to case 1 and thus are done. If not, the non-containment index decreases (by the same logic as 
before), and we are again done.

\textit{Case 2b.}
$\begin{pmatrix}
12 & 13 & 2 & 2\\
1 & 3 & {\bf 1}2 & 13 
\end{pmatrix}$
Possible singletons: 
$\begin{pmatrix}
1 & 2 & 3 & 1 & 2 & 2\\
1 & 2 & 3 & 3 & 1 & 3
\end{pmatrix}$

\smallskip

\textit{Case 2c.}
$\begin{pmatrix}
12 & 13 & 2 & 3\\
1 & 1 & 12 & 13 
\end{pmatrix}$
Possible singletons: 
$\begin{pmatrix}
1 & 2 & 3 & 1 & 1 & 2 & 3\\
1 & 2 & 3 & 2 & 3 & 3 & 2
\end{pmatrix}$

At this point we are not ready to deal with case 2c; we revisit it after case 2m.


\textit{Case 2d.}
$\begin{pmatrix}
12 & 13 & 3 & 3\\
1 & 1 & 12 & \textbf{2}3 
\end{pmatrix}$
Possible singletons: 
$\begin{pmatrix}
1 & 2 & 3 & 2 & 3 & 3\\
1 & 2 & 3 & 1 & 1 & 2
\end{pmatrix}$

\textit{Case 2e.}
$\begin{pmatrix}
12 & 13 & 2 & 2\\
1 & 1 & {\bf 1}2 & 23 
\end{pmatrix}$
Possible singletons: 
$\begin{pmatrix}
1 & 2 & 3 & 2 & 2 & 3\\
1 & 2 & 3 & 1 & 3 & 1
\end{pmatrix}$

Eliminating the boldface {\bf 1} changes some 1 to $1x$. If it changes position 1 to 12 or position 2 to 13, we reduce to case 1. It can't change position 
1 to 13 or position 2 to 12 because of incomparability. The only other case where the non-containment index does not decrease is if it changes 
$\begin{pmatrix}
2 \\
1
\end{pmatrix}$ to
$\begin{pmatrix}
2 \\
13
\end{pmatrix}$, but this is a reduction to a situation isomorphic to case 2d.

\textit{Case 2f.}
$\begin{pmatrix}
12 & 13 & 2 & 3\\
1 & 1 & {\bf 1}2 & 23 
\end{pmatrix}$
Possible singletons: 
$\begin{pmatrix}
1 & 2 & 3 & 2 & 3 & 3\\
1 & 2 & 3 & 1 & 1 & 2
\end{pmatrix}$


\textit{Case 2g.}
$\begin{pmatrix}
12 & 13 & 1 & 3\\
2 & 1 & 12 & {\bf 2}3 
\end{pmatrix}$
Possible singletons: 
$\begin{pmatrix}
1 & 2 & 3 & 1 & 3 & 3\\
1 & 2 & 3 & 2 & 1 & 2
\end{pmatrix}$

\textit{Case 2h.}
$\begin{pmatrix}
12 & 13 & 3 & 3\\
2 & 1 & 12 & {\bf 2}3 
\end{pmatrix}$
Possible singletons: 
$\begin{pmatrix}
1 & 2 & 3 & 1 & 3 & 3\\
1 & 2 & 3 & 2 & 1 & 2
\end{pmatrix}$

\textit{Case 2i.}
$\begin{pmatrix}
{\bf 1}2 & 13 & 1 & 1\\
2 & 2 & 12 & 23 
\end{pmatrix}$
Possible singletons: 
$\begin{pmatrix}
1 & 2 & 3 & 1 & 1 & 3\\
1 & 2 & 3 & 2 & 3 & 2
\end{pmatrix}$

\textit{Case 2j.}
$\begin{pmatrix}
{\bf 1}2 & 13 & 1 & 3\\
2 & 2 & 12 & 23 
\end{pmatrix}$
Possible singletons: 
$\begin{pmatrix}
1 & 2 & 3 & 1 & 1 & 3 & 3\\
1 & 2 & 3 & 2 & 3 & 1 & 2
\end{pmatrix}$

At most one of 
$\begin{pmatrix} 1\\3\end{pmatrix}$
and
$\begin{pmatrix} 3\\1\end{pmatrix}$
can occur, but this does not affect the argument: removing the boldface {\bf 1} adds a 2 somewhere, which must reduce the non-containment index.

\textit{Case 2k.}
$\begin{pmatrix}
{\bf 1}2 & 13 & 3 & 3\\
2 & 2 & 12 & 23 
\end{pmatrix}$
Possible singletons: 
$\begin{pmatrix}
1 & 2 & 3 & 1 & 3 & 3\\
1 & 2 & 3 & 2 & 1 & 2
\end{pmatrix}$

\textit{Case 2l.}
$\begin{pmatrix}
12 & 13 & 1 & 3\\
2 & 3 & 1{\bf 2} & 23 
\end{pmatrix}$
Possible singletons: 
$\begin{pmatrix}
1 & 2 & 3 & 1 & 1 & 3\\
1 & 2 & 3 & 2 & 3 & 2
\end{pmatrix}$

If the boldface {\bf 2}'s removal produced 123 in position 4, we would have a problem, but this is not possible due to comparability.

Now we are ready to deal with case 2c.

\textit{Case 2c.}
$\begin{pmatrix}
1\textbf{2} & 13 & 2 & 3\\
1 & 1 & 12 & 13 
\end{pmatrix}$
Possible singletons: 
$\begin{pmatrix}
1 & 2 & 3 & 2 & 3 & 2 & 3\\
1 & 2 & 3 & 1 & 1 & 3 & 2
\end{pmatrix}$

Removing the boldface {\bf 2} changes some 2 to $2x$. If this is in position 3, we are done by reduction to case 1. Otherwise, the only case in which this does not produce 
an extra containment relation is when we change 
$\begin{pmatrix} 2\\2\end{pmatrix}$ to
$\begin{pmatrix} 23\\2\end{pmatrix}$ or
$\begin{pmatrix} 2\\1\end{pmatrix}$ to
$\begin{pmatrix} 23\\1\end{pmatrix}$. The results are isomorphic to cases 2l and 2j, respectively, so we are done.

\smallskip

\textbf{Case 3.} At least one of the types comes from a triangle (type $A$) vertex.

We can assume that the matrix contains
$\begin{pmatrix}
{\bf 1}23 & 1 \\
\alpha & \beta  
\end{pmatrix},$
where $\alpha$ and $\beta$ are subsets of $\{1,2,3\}$ and $\beta$ is not a singleton. If $i=1$ we are done. Otherwise remove the boldface {\bf 1} from $A$, changing some 1 to $1x$. Notice that $\alpha$ cannot contain $1$ by comparability, so the type $A'$ obtained is such that each set in $\{A'_j, B_j, A'_j \cup B_j\}$ contains one of the sets in $\{A_j,B_j,A_j \cup B_j\}$ (for all $j \neq i$), and $A'_i \cup B_i \supseteq A_i \cup B_i$. This transformation might not reduce the non-containment index, but it makes the triangle $A$ into a rhombus $A'$. If $B$ is also a triangle, we can make it a rhombus in the same way, and then invoke cases 1 and 2.

\medskip

This completes the typology and thus the proof.
\end{proof}

\end{document}